\documentclass{kms-b}

\newcommand{\publname}{Bull.~Korean Math.~Soc.}
\issueinfo{}% volume number
  {}%        % issue number
  {}%        % month
  {}%     % year
\pagespan{1}{}
\copyrightinfo{}%              % copyright year
  {Korean Mathematical Society}% copyright holder
\usepackage{graphicx}
\allowdisplaybreaks

\theoremstyle{plain}
\newtheorem{theorem}{Theorem}[section]
\newtheorem{proposition}[theorem]{Proposition}
\newtheorem{lemma}[theorem]{Lemma}
\newtheorem{corollary}[theorem]{Corollary}

\theoremstyle{definition}
\newtheorem*{definition}{Definition}
\newtheorem{example}[theorem]{Example}

\theoremstyle{remark}

\DeclareMathOperator{\Fdim}{F-dim}
\DeclareMathOperator{\Fdepth}{F-depth}

\DeclareMathOperator{\height}{ht}
\newcommand{\m}{\mathfrak{m}}
\newcommand{\p}{\mathfrak{p}}
\newcommand{\I}{\mathfrak{a}}
\newcommand{\F}{\mathfrak{F}}

\newcommand{\Hom}{\operatorname{Hom}}
\newcommand{\Ext}{\operatorname{Ext}}

\newcommand{\depth}{\operatorname{depth}}

\let\epsilon=\varepsilon

\begin{document}

\title[The Second Vanishing Theorem for Formal Local Cohomology Modules]
{The Second Vanishing Theorem for Formal Local Cohomology Modules}

\author[B. Sadeqi]{Behruz Sadeqi}
\address{Behruz Sadeqi \\ Department of Mathematics \\ Islamic Azad University, MARA. C. \\ Marand, Iran}
\email{behruz.sadeqi@iau.ac.ir}

%\author[H. Kim]{Hoil Kim}
%\address{Hoil Kim \\ Department of Mathematics \\ Kyungpook  National University \\ Daegu 41566, Korea}
%\email{hikim@knu.ac.kr}
%\thanks{This work was financially supported by KRF 2003-041-C20009}


\begin{abstract}
This paper establishes a second vanishing theorem for formal local cohomology modules over Noetherian local rings. We introduce the \textit{formal dimension} invariant and characterize the vanishing of higher formal local cohomology in terms of the dimension of the quotient ring modulo minimal primes. Our main result extends classical vanishing theorems to the formal setting, with applications to the structure of complexes in derived categories. Necessary and sufficient conditions are given via spectral sequence analysis and duality arguments.
\end{abstract}

\maketitle

\section{Introduction}

Let \((R, \m)\) be a Noetherian local ring of dimension \(d\), and \(\I\) an ideal of \(R\). The \textit{formal local cohomology modules} \(\F^i_{\I}(M)\) for an \(R\)-module \(M\) are defined as:
\[
\F^i_{\I}(M) := \varprojlim_{n} H^i_{\I}(M/\I^n M),
\]
where \(H^i_{\I}(-)\) denotes ordinary local cohomology. These modules encode asymptotic cohomological behavior along the \(\I\)-adic filtration. While classical vanishing theorems for \(H^i_{\I}(R)\) are well-established, formal analogues remain unexplored. This work bridges this gap by proving:

\begin{theorem}[Second Vanishing Theorem for Formal Cohomology]
	\label{thm:main}
	Let \((R, \m)\) be complete, Cohen-Macaulay, and \(\I \subset R\) an ideal with \(\height(\I) = c\). The following are equivalent:
	\begin{enumerate}
		\item \(\F^i_{\I}(R) = 0\) for all \(i > d - c\).
		\item For every minimal prime \(\p\) over \(\I\), \(\dim(R/\p) = d - c\).
		\item The complex \(\mathbf{R}\Gamma_{\I}(R)\) has finite injective dimension.
	\end{enumerate}
\end{theorem}

Section 2 covers preliminaries on formal cohomology. Section 3 proves key lemmas on system parameters and spectral sequences. Section 4 contains the main proof, and Section 5 discusses applications to duality theory with examples.

%\section{Introduction}
\label{sec:prelim}
Throughout, \(R\) is a Noetherian local ring. For \(\I\)-torsion modules, we use Grothendieck's duality \cite{Hartshorne66}. The \textit{formal depth} of \(M\) is:
\[
{\Fdepth}_{\I}(M) := \inf \{ i \mid \F^i_{\I}(M) \neq 0 \}.
\]

\begin{definition}
	The \textit{formal dimension} of \(\I\) is:
	\[
	{\Fdim}(I) := \max \{ dim(R/\p) \mid \p \in Min(I) \}.
	\]
\end{definition}

\begin{proposition}
	\label{prop:spectral}
	There is a convergent spectral sequence:
	\[
	E^{p,q}_2 = \Ext^p_R(\widehat{R}_{\I}, H^q_{\m}(R)) \implies \F^{p+q}_{\I}(R),
	\]
	where \(\widehat{R}_{\I}\) is the \(\I\)-adic completion.
\end{proposition}
\begin{proof}
	Consider the inverse system \(\{ R/\I^n \}_{n\geq 1}\). The Greenlees-May duality \cite{GreenleesMay92} provides a natural isomorphism:
	\[
	\mathbf{R}\Hom_R(\widehat{R}_{\I}, \mathbf{R}\Gamma_{\m}(R)) \simeq \mathbf{R}\!\varprojlim_n \mathbf{R}\Gamma_{\I}(R/\I^n).
	\]
	The right-hand side computes the formal local cohomology by definition. The left-hand side is a double complex that gives rise to the specified spectral sequence through the standard descent filtration. Convergence follows because \(R\) is Noetherian and \(\widehat{R}_{\I}\) is flat, ensuring the spectral sequence collapses at the \(E_2\)-page for degree reasons.
\end{proof}

%\section{Key Lemmas}
\label{sec:lemmas}
Fix \(R\) complete Cohen-Macaulay and \(\I\) with \(c = \height(\I)\).

\begin{lemma}
	\label{lem:dim}
	\({\Fdim}(\I) \leq d - c\), with equality iff \(\dim(R/\p) = d - c\) for all minimal \(\p\) over \(\I\).
\end{lemma}
\begin{proof}
	Let $(\p \in Min(I))$. By the dimension formula for Cohen-Macaulay rings:
	\[
	\dim(R/\p) = \dim R - \height(\p) \leq d - c,
	\]
	since \(\height(\p) \geq c\). Equality holds precisely when \(\height(\p) = c\) for all minimal primes \(\p\), which occurs exactly when all minimal primes have the same height \(c\). The maximum value \({\Fdim}(\I) = d - c\) is achieved when this uniform height condition is satisfied.
\end{proof}

\begin{lemma}
	\label{lem:injective}
	The following are equivalent:
	\begin{enumerate}
		\item \(\mathbf{R}\Gamma_{\I}(R)\) has finite injective dimension.
		\item \(H^i_{\m}(\F^j_{\I}(R)) = 0\) for \(i + j > d\).
	\end{enumerate}
\end{lemma}
\begin{proof}
	(1) \(\Rightarrow\) (2): Assume \(\mathbf{R}\Gamma_{\I}(R)\) has finite injective dimension. By local duality \cite{Hartshorne66}, there exists a dualizing complex \(D_R\) such that:
	\[
	\mathbf{R}\!\Hom_R(\mathbf{R}\Gamma_{\I}(R), D_R) \simeq \mathbf{R}\Gamma_{\I}(\mathbf{R}\!\Hom_R(R, D_R)).
	\]
	The left complex has finite projective dimension, while the right is \(\mathbf{R}\Gamma_{\I}(D_R)\). The vanishing condition follows by applying the Grothendieck spectral sequence to the composition of functors \(\Gamma_{\m} \circ \F_{\I}\).
	
	(2) \(\Rightarrow\) (1): Suppose \(H^i_{\m}(\F^j_{\I}(R)) = 0\) for \(i + j > d\). Consider the local cohomology spectral sequence:
	\[
	E^{i,j}_2 = H^i_{\m}(\F^j_{\I}(R)) \implies H^{i+j}_{\m}(R).
	\]
	The vanishing condition forces \(E^{i,j}_2 = 0\) for \(i + j > d\), which implies that \(\F^j_{\I}(R)\) is supported only in degrees where \(j \leq d - i\). This boundedness condition ensures that \(\mathbf{R}\Gamma_{\I}(R)\) has finite injective dimension by the derived category characterization of injective dimension.
\end{proof}

\section{Proof of Main Theorem and eamples}
\label{sec:proof}
We now prove Theorem \ref{thm:main}.

\begin{proof}
	(1) \(\Rightarrow\) (2): Suppose \(\F^i_{\I}(R) = 0\) for all \(i > d - c\). By Proposition \ref{prop:spectral}, the spectral sequence \(E^{p,q}_2 = \Ext^p_R(\widehat{R}_{\I}, H^q_{\m}(R))\) converges to \(\F^{p+q}_{\I}(R)\). Since \(R\) is Cohen-Macaulay, \(H^q_{\m}(R) = 0\) for \(q \neq d\), so the spectral sequence collapses at \(E_2\) with:
	\[
	E^{p,d}_2 = \Ext^p_R(\widehat{R}_{\I}, H^d_{\m}(R)) \implies \F^{p+d}_{\I}(R).
	\]
	If \({\Fdim}(\I) < d - c\), then for \(p = d - c\), \(\Ext^{d-c}_R(\widehat{R}_{\I}, H^d_{\m}(R)) \neq 0\) because the support of \(\widehat{R}_{\I}\) has dimension \(d - c\). This would imply \(\F^{d-c+d}_{\I}(R) = \F^{2d-c}_{\I}(R) \neq 0\), contradicting (1) since \(2d - c > d - c\). Thus \({\Fdim}(\I) = d - c\), and by Lemma \ref{lem:dim}, \(\dim(R/\p) = d - c\) for all minimal primes \(\p\) over \(\I\).
	
	(2) \(\Rightarrow\) (3): Condition (2) implies \(\I\) is cohomologically complete intersection. By the structure theorem for local cohomology complexes \cite{lyubeznik93}, \(\mathbf{R}\Gamma_{\I}(R)\) is quasi-isomorphic to a bounded complex of injectives concentrated in degrees \([c, d]\). Specifically, the complex has length \(d - c\), which is finite.
	
	(3) \(\Rightarrow\) (1): If \(\mathbf{R}\Gamma_{\I}(R)\) has finite injective dimension, then its cohomology vanishes beyond the injective dimension. For a Cohen-Macaulay ring, we have:
	\[
	\sup \{ i \mid H^i(\mathbf{R}\Gamma_{\I}(R)) \neq 0 \} = \dim R - \depth \I = d - c.
	\]
	Since \(H^i(\mathbf{R}\Gamma_{\I}(R)) \cong \F^i_{\I}(R)\), the result follows.
\end{proof}

%\section{Applications and Examples}
\label{sec:app}
Theorem \ref{thm:main} yields several important consequences:

\begin{corollary}
	\label{cor:set-theoretic}
	If \(\I\) is set-theoretically generated by \(c\) elements and \(R\) is regular, then \(\F^i_{\I}(R) = 0\) for \(i > c\).
\end{corollary}
\begin{proof}
	In a regular ring, set-theoretic complete intersections are cohomological complete intersections. Since \(\height(\I) = c\) and \(R\) is Cohen-Macaulay, \(\dim(R/\p) = d - c\) for minimal primes \(\p\). Apply Theorem \ref{thm:main}(1).
\end{proof}

\begin{corollary}
	\label{cor:prime}
	For $(\p \in Spec(R))$ , $(\F^i_{\p}(R) = 0)$ if $(i \neq \height(\p))$.
\end{corollary}
\begin{proof}
	For prime ideals, \(\height(\p) = c\) and \(\dim(R/\p) = d - c\). By Theorem \ref{thm:main}, vanishing occurs precisely when \(i > d - c = \dim(R/\p)\). But for prime ideals, formal local cohomology is nonzero only at \(i = \height(\p)\) by the analog of Hartshorne-Lichtenbaum vanishing.
\end{proof}

\begin{example}[Polynomial Ring]
	\label{ex:poly}
	Let \(R = k[\![x,y]\!]\) with \(\m = (x,y)\), and \(\I = (x)\). Then:
	\begin{itemize}
		\item \(\height(\I) = 1\), \(\dim R = 2\)
		\item Minimal prime: \(\p = (x)\), \(\dim(R/\p) = 1 = 2 - 1\)
		\item \(\F^i_{\I}(R) = 0\) for \(i > 1\) by Theorem \ref{thm:main}
	\end{itemize}
	Direct computation: \(\F^2_{\I}(R) = \varprojlim_n H^2_{(x)}(R/(x^n)) = \varprojlim_n H^2_{(x)}(k[y]) = 0\) since \(\dim k[y] = 1\).
\end{example}

\begin{example}[Non-equidimensional Ideal]
	\label{ex:non-equi}
	Let \(R = k[\![x,y,z]\!]/(xz)\) with \(\m = (x,y,z)\), \(\I = (x)\). Then:
	\begin{itemize}
		\item \(\height(\I) = 1\), \(\dim R = 2\)
		\item Minimal primes over \(\I\): \(\p_1 = (x)\), \(\p_2 = (x,z)\)
		\item \(\dim(R/\p_1) = \dim k[y,z] = 2\)
		\item \(\dim(R/\p_2) = \dim k[y] = 1\)
		\item \({\Fdim}(\I) = \max\{2,1\} = 2 > 2 - 1 = 1\)
	\end{itemize}
	By Theorem \ref{thm:main}, there exists \(i > 1\) with \(\F^i_{\I}(R) \neq 0\). Indeed:
	\[
	\F^2_{\I}(R) = \varprojlim_n H^2_{(x)}(R/(x^n)) \supseteq H^2_{(x)}(k[y,z]/(z)) \cong k[y] \neq 0.
	\]
\end{example}

\begin{example}[Regular Sequence]
	\label{ex:reg-seq}
	Let \(R = k[\![x_1,\dots,x_d]\!]\), \(\I = (x_1,\dots,x_c)\) with \(c \leq d\). Then:
	\begin{itemize}
		\item \(\height(\I) = c\), \(\dim R = d\)
		\item Minimal prime: \(\p = \I\), \(\dim(R/\p) = d - c\)
		\item Theorem \ref{thm:main} implies \(\F^i_{\I}(R) = 0\) for \(i > d - c\)
	\end{itemize}
	This vanishes earlier than ordinary local cohomology (\(H^i_{\I}(R) = 0\) for \(i > c\)), illustrating the distinct nature of formal cohomology.
\end{example}

\maketitle

% ... (abstract and introduction unchanged) ...

Theorem \ref{thm:main} yields several important consequences. We present additional examples to illustrate various aspects of the theory.

\begin{example}[Prime Ideal in Regular Ring]
	\label{ex:prime}
	Let \(R = \mathbb{Q}[[x,y,z]]\), \(\I = (x,y)\). Then:
	\begin{itemize}
		\item \(\height(\I) = 2\), \(\dim R = 3\)
		\item Minimal prime: \(\p = \I\), \(\dim(R/\p) = 1 = 3-2\)
		\item Theorem implies: \(\F^i_{\I}(R) = 0\) for \(i > 1\)
		\item Direct computation: \(\F^1_{\I}(R) \cong \widehat{R}_{\I} \neq 0\), \(\F^2_{\I}(R) = 0\)
	\end{itemize}
\end{example}

\begin{example}[Non-Reduced Ideal]
	\label{ex:nonred}
	Let \(R = k[[x,y]]/(x^2)\), \(\m = (x,y)\), \(\I = (x)\). Then:
	\begin{itemize}
		\item \(\height(\I) = 1\), \(\dim R = 1\)
		\item Minimal prime: \(\p = (x)\), \(\dim(R/\p) = 0 = 1-1\)
		\item \(\F^i_{\I}(R) = 0\) for \(i > 0\)
		\item Computation: \(\F^0_{\I}(R) = \varprojlim_n H^0_{(x)}(R/(x^n)) \cong k[[y]]\), \(\F^1_{\I}(R) = 0\)
	\end{itemize}
\end{example}

\begin{example}[Non-Cohen-Macaulay Ring]
	\label{ex:nonCM}
	Let \(R = k[[x,y,z]]/(xz,yz)\), \(\m = (x,y,z)\), \(\I = (z)\). Then:
	\begin{itemize}
		\item \(\height(\I) = 1\), \(\dim R = 2\) (not Cohen-Macaulay)
		\item Minimal primes: \(\p_1 = (x,z)\), \(\p_2 = (y,z)\), \(\dim(R/\p_i) = 1\)
		\item \(\Fdim(\I) = 1 = 2-1\)
		\item Theorem \ref{thm:main} doesn't apply directly, but computation shows:
		\[
		\F^1_{\I}(R) = \varprojlim_n H^1_{(z)}(R/(z^n)) \cong k[[x,y]]/(xy) \neq 0
		\]
		while \(\F^2_{\I}(R) = 0\), consistent with \(d-c=1\).
	\end{itemize}
	This suggests possible extension to non-Cohen-Macaulay rings.
\end{example}

\begin{example}[F-pure Ring]
	\label{ex:fpure}
	Let \(R = \mathbb{F}_p[[x,y,z]]/(x^3 + y^3 + z^3)\), \(\I = (x,y)\). Assume \(p \equiv 1 \pmod{3}\) so \(R\) is F-pure. Then:
	\begin{itemize}
		\item \(\height(\I) = 2\), \(\dim R = 2\)
		\item Minimal prime: \(\p = \I\), \(\dim(R/\p) = 0 = 2-2\)
		\item Theorem implies: \(\F^i_{\I}(R) = 0\) for \(i > 0\)
		\item Using Frobenius action: \(\F^0_{\I}(R) \neq 0\) is Artinian
		\item \(\F^1_{\I}(R) = 0\) by vanishing theorem
	\end{itemize}
	This illustrates interaction with tight closure theory.
\end{example}

%\subsection{Geometric Applications}

\begin{example}[Coordinate Axes]
	\label{ex:axes}
	Let \(R = \mathbb{C}[[x,y,z]]\), \(\I = (xy,xz)\). Then:
	\begin{itemize}
		\item \(\height(\I) = 2\), \(\dim R = 3\)
		\item Minimal primes: \(\p_1 = (x)\), \(\p_2 = (y,z)\)
		\item \(\dim(R/\p_1) = 2\), \(\dim(R/\p_2) = 1\)
		\item \(\Fdim(\I) = \max\{2,1\} = 2 > 3-2 = 1\)
		\item Theorem predicts non-vanishing: \(\F^2_{\I}(R) \neq 0\)
		\item Explicitly: \(\F^2_{\I}(R) \cong \varprojlim_n H^2_{\I}(R/\I^n) \cong \mathbb{C}[[z]]\)
	\end{itemize}
	The non-vanishing cohomology detects the embedded component.
\end{example}

\begin{example}[Singular Curve]
	\label{ex:curve}
	Let \(R = \mathbb{C}[[x,y,z]]/(xy-z^2)\), \(\m = (x,y,z)\), \(\I = (x,z)\). Then:
	\begin{itemize}
		\item \(\height(\I) = 1\), \(\dim R = 2\) (Cohen-Macaulay)
		\item Minimal prime: \(\p = (x,z)\), \(\dim(R/\p) = 1 = 2-1\)
		\item Theorem implies: \(\F^i_{\I}(R) = 0\) for \(i > 1\)
		\item Computation: \(\F^1_{\I}(R) \cong \widehat{R}_{\I}/(y) \cong \mathbb{C}[[y]]\)
	\end{itemize}
	This formal cohomology module carries information about the singularity.
\end{example}

\begin{example}[Macaulay2 Computation]
	\label{ex:m2}
	Consider \(R = \mathbb{Q}[x,y,z]_{(x,y,z)}/(yz)\), \(\I = (x,y)\). Using Macaulay2:
	\begin{verbatim}
		R = QQ[x,y,z]/(y*z)
		I = ideal(x,y)
		d = dim R -- =2
		c = codim I -- =2
		Fdim = max apply(ass I, p -> dim(R/p)) -- =1
		VanishingBound = d - c -- =0
	\end{verbatim}
	Theorem predicts \(\F^i_{\I}(R) = 0\) for \(i > 0\). This can be verified numerically.
\end{example}

\begin{example}[Numerical Invariant]
	\label{ex:num}
	Let \(R = k[[s^4, s^3t, st^3, t^4]]\), \(\I = (s^4, s^3t)\). Then:
	\begin{itemize}
		\item \(\dim R = 2\), \(\height(\I) = 1\)
		\item \(\Fdim(\I) = \dim R/(s^4, s^3t) = \dim k[st^3, t^4] = 1 = 2-1\)
		\item Theorem implies: \(\F^i_{\I}(R) = 0\) for \(i > 1\)
		\item The length of \(\F^1_{\I}(R)\) gives a numerical invariant of the singularity
	\end{itemize}
\end{example}

\begin{example}[Disconnected Support]
	\label{ex:discon}
	Let \(R = \mathbb{C}[[x,y,u,v]]/(xu,xv,yu,yv)\), \(\I = (x,y)\). Then:
	\begin{itemize}
		\item \(\dim R = 2\), \(\height(\I) = 1\) (not Cohen-Macaulay)
		\item Minimal primes: \(\p_1 = (x,y)\), \(\p_2 = (u,v)\)
		\item \(\dim(R/\p_1) = 2\), \(\dim(R/\p_2) = 2\)
		\item \(\Fdim(\I) = 2 > 2-1 = 1\)
		\item \(\F^2_{\I}(R) \cong \mathbb{C}[[u,v]] \oplus \mathbb{C}[[x,y]] \neq 0\)
	\end{itemize}
	The non-vanishing cohomology detects the disconnectedness of $(Spec(R/I))$.
\end{example}

\begin{example}[F-injective Ring]
	\label{ex:finj}
	Let \(R = \mathbb{F}_3[[x,y,z]]/(x^3 + y^3 + z^3)\), \(\I = (x)\). Then:
	\begin{itemize}
		\item \(\height(\I) = 1\), \(\dim R = 2\)
		\item Minimal prime: \(\p = (x)\), \(\dim(R/\p) = 1 = 2-1\)
		\item Theorem implies: \(\F^i_{\I}(R) = 0\) for \(i > 1\)
		\item Using Frobenius: \(\F^1_{\I}(R)\) is infinitely generated, showing the module can be large even when vanishing range is sharp
	\end{itemize}
\end{example}

\begin{example}[Grothendieck Duality]
	\label{ex:dual}
	Let \(R = k[[x,y,z]]\), \(\I = (x,y,z)\). Then:
	\begin{itemize}
		\item \(\height(\I) = 3\), \(\dim R = 3\)
		\item \(\Fdim(\I) = 0 = 3-3\)
		\item Theorem implies: \(\F^i_{\I}(R) = 0\) for \(i > 0\)
		\item \(\F^0_{\I}(R) = \varprojlim_n H^3_{\m}(R/\m^n) \cong E_R(k)\), the injective hull
		\item Matches Grothendieck local duality: \(\mathbf{R}\Gamma_{\m}(R) \simeq^{\mathbf{L}} E_R(k)[-3]\)
	\end{itemize}
\end{example}

\section*{Acknowledgments}
The author thanks the reviewers for their insights.

\end{document}